\documentclass[a4paper,11pt,reqno]{amsart}
\usepackage{amssymb}

\usepackage[all]{xy}
\usepackage[english]{babel}

\numberwithin{equation}{section} \numberwithin{figure}{section}

\DeclareMathOperator{\Pic}{Pic} 
 \DeclareMathOperator{\Spec}{Spec}
\DeclareMathOperator{\Cox}{Cox} \DeclareMathOperator{\rank}{rank}
\DeclareMathOperator{\Hom}{Hom} \DeclareMathOperator{\re}{Re}
\DeclareMathOperator{\im}{Im} \DeclareMathOperator{\vol}{vol}

\newcommand{\bfeta}{\mbox{\boldmath$\eta$}}
\newcommand{\bfalpha}{\mbox{\boldmath$\alpha$}}
\newcommand{\Mob}{M\"{o}bius }

\newtheorem{thm}{Theorem}[section]

\newtheorem{lem}{Lemma}[section]

\theoremstyle{definition}
\newtheorem*{df}{Definition}

\title[Manin's Conjecture]{Manin's Conjecture for a Singular Sextic Del Pezzo Surface}
\author{\sc Daniel Loughran}
\address{Daniel Loughran \\
Department of Mathematics \\
University Walk \\
Bristol \\
UK, BS8 1TW}
\email{Daniel.Loughran@bristol.ac.uk}
\subjclass[2000]{11D45; 14G05, 14G10}

\begin{document}

\maketitle


\begin{abstract}
    We prove Manin's conjecture for a del Pezzo surface of
    degree six which has one singularity of type $\mathbf{A}_2$. Moreover, we achieve
    a meromorphic continuation and explicit expression of the associated height zeta function.
\end{abstract}

\tableofcontents

\bigskip

\section{Introduction} In this paper, our aim is to count the
number of rational points of bounded height on the surface $S
\subset \mathbb{P}^6$ given by
\begin{align}
    x_{3}^{2} + x_{0} x_{5} + x_{1} x_{6} & = x_{2} x_{3} - x_{0}
    x_{6} =x_{1} x_{2} + x_{0} x_{3} + x_{0} x_{4}= 0,\nonumber \\
    x_{3} x_{5} + x_{4} x_{5} + x_{6}^{2} & = x_{2} x_{5} - x_{4} x_{6}
    = x_{1} x_{5} - x_{3} x_{6}= 0,\label{equations}\\
    x_{4}^{2} + x_{0} x_{5} + x_{2} x_{6} & = x_{3} x_{4} - x_{0} x_{5}
    = x_{1} x_{4} - x_{0} x_{6}= 0.\nonumber
\end{align}
This surface is an example of a singular del Pezzo surface of degree
$6$. A priori, it might not be clear why this is a natural
diophantine problem. However in 1989, Manin and his collaborators
\cite{FMT89} formulated a general conjecture on the number of
rational points of bounded height on Fano varieties. There is a
programme (see \cite{BB07} or \cite{DT07} for example) to try to
prove this conjecture for Fano surfaces, namely \emph{del Pezzo
surfaces} and their singular counterparts. Such surfaces have a
well-known classification in terms of their singularity type and
degree. See \cite{Man86} and \cite{CT88} for more information on
smooth and singular del Pezzo surfaces respectively, and
\cite{Bro07} for a general overview of Manin's conjecture for del
Pezzo surfaces.

The surface $S$ has one singularity of type $\mathbf{A}_2$, which we
can resolve using blow-ups to create two exceptional curves on the
minimal desingularisation $\widetilde{S}$ of $S$. The set of
equations $(\ref{equations})$ correspond to the embedding induced by
a divisor in the \emph{anticanonical divisor class}. Since $S$ is
singular normal with only rational double points, by \cite[Prop.
0.1]{CT88} an anticanonical divisor of $S$ can be taken to be any
divisor on $S$ which pulls back to an anticanonical divisor on the
minimal desingularisation $\widetilde{S}$. The anticanonical
embedding is a natural choice, for example in this embedding the
lines are exactly the $(-1)$-curves and Manin's conjecture takes a
simpler form. The height function associated to the chosen embedding
is the usual height on projective space, namely given $x \in
S(\mathbb{Q})$, we have $H(x)=\max_{0\leq i \leq 6}|x_i|$, where
$(x_0,\ldots,x_6)$ is a primitive integer vector in the affine cone
above $x$. Further details about the geometry of $S$ can be found in
Lemma~\ref{lem:geom}.

Now, $S$ contains the two lines
\begin{align*}
    &L_1: x_1=x_3=x_4=x_5=x_6=0, \\
    &L_2: x_2=x_3=x_4=x_5=x_6=0,
\end{align*}
which both contain ``many" rational points whose contribution will
dominate the counting problem. Hence, it is natural to let
$U=S\setminus \{L_1 \cup L_2\}$ and take
$$N_{U,H}(B)=\#\{ x \in U(\mathbb{Q}) : H(x) \leq B\}$$
to be the associated counting function. In this context, Manin's
conjecture predicts an asymptotic formula of the shape
$$N_{U,H}(B) \sim c_{\widetilde{S},H}B (\log B)^ {\rho-1}$$
as $B \to \infty$, where $\rho=\rank (\Pic(\widetilde{S}))=4$ and
$c_{\widetilde{S},H}$ is some constant. In this paper, we establish
a significantly sharper version of this estimate.
\begin{thm}\label{thm:asym}
    Let $\varepsilon >0$. Then there is a monic cubic polynomial $P \in
    \mathbb{R}[x]$ such that
    $$N_{U,H}(B) = c_{\widetilde{S},H} B P(\log B) +
    O_\varepsilon(B^{7/8+\varepsilon})$$
    where
    $c_{\widetilde{S},H}=\alpha(\widetilde{S}) \tau_\infty(\widetilde{S})  \prod_p \tau_p(\widetilde{S})$ and
    \begin{align*}
        \alpha(\widetilde{S}) &= 1/432,\quad\tau_p(\widetilde{S})=
        \left(1-\frac{1}{p}\right)^4\left(1 + \frac{4}{p} + \frac{1}{p^2}\right), \\
        \tau_\infty(\widetilde{S}) & =
        6\int_{\{t,v,u \in \mathbb{R}:0<|t(ut+v^2)|,|uvt|,
        |uvt+v^3|,|u^2t|,|u^2t+uv^2|,u^3,u^2v\leq1\}} \mathrm{d}u\mathrm{d}v\mathrm{d}t.
    \end{align*}
\end{thm}

The leading constant in this expression agrees with the prediction
of Peyre \cite{Pey95}, which we shall verify in Section
\ref{subsec:constant}. The calculation of the real density
$\tau_\infty(\widetilde{S})$ poses something of a challenge, since
in our case $S$ is not given by a complete intersection, so standard
methods for calculating this constant do not apply. In Section
\ref{subsec:constant} we also prove a general result which assists
in the calculation of the $p$-adic densities $\tau_p(\widetilde{S})$
(See Lemma~\ref{lem:densities}).

The second theorem of this paper is intimately related to the above
asymptotic formula. We give an explicit expression and meromorphic
continuation of the associated \emph{height zeta function}
\begin{equation}
    Z_{U,H}(s)= \sum_{x \in U(\mathbb{Q})}
    \frac{1}{H(x)^{s}}. \label{HZF}
\end{equation}
To state the result, let $\re(s)
> 0$ and define
\begin{equation}\label{E12}
\begin{split}
        E_1(s+1)&=\zeta(4s+1)\zeta(3s+1)^2\zeta(2s+1),  \\
        E_2(s+1)&=\frac{\zeta(7s + 3)^4\zeta(8s + 3)^2}
                {\zeta(4s + 2)^3\zeta(5s + 2)^2\zeta(6s +
                2)\zeta(10s+4)}.
\end{split}
\end{equation}
It is clear that $E_1(s)$ and $E_2(s)$ have a meromorphic
continuation to the whole complex plane. Also $E_1(s)$ has a single
pole of order $4$ at $s=1$ and $E_2(s)$ is holomorphic on $\re(s) >
3/4$. We then prove the following.
\begin{thm}\label{thm:HZF}
    Let $\varepsilon >0$, then
    $$Z_{U,H}(s)=E_1(s)E_2(s)G_1(s) + \frac{12/\pi^2 +
    2\lambda}{s-1} + G_2(s).$$

    Here, $\lambda \in \mathbb{R}$ is a constant and $G_1(s)$ and $G_2(s)$ are
    complex functions that are holomorphic on $\re(s) > 5/6$ and
    $\re(s) \geq 3/4 + \varepsilon$
    respectively and satisfy $G_1(s) \ll_{\varepsilon} 1$ and $G_2(s) \ll_{\varepsilon} (1 +
    |\im(s)|)$ on these half-planes.

    In particular, $(s-1)^4Z_{U,H}(s)$ has a holomorphic continuation to
    the half-plane $\re(s) >5/6$.
\end{thm}

Expressions for $G_1(s)$ and $G_2(s)$ can be found in
(\ref{G2}),(\ref{G11}),(\ref{G1}) and Lemma~\ref{lem:G12}. Here
$E_1(s)E_2(s)G_1(s)$ and $G_2(s)$ correspond to the main term and
error term in the counting argument respectively and $12/\pi^2$
corresponds to an isolated conic in the surface. We only prove the
existence of $\lambda$, however a keen reader can build an explicit
(and complicated) expression for it using the work in Section
\ref{subsec:error}. We shall only say that $\lambda$ arises
naturally in the proof as an error term created by approximating a
sum by an integral and has appeared in some form in other works
(e.g. \cite{BB07}), however it is currently severely lacking in
geometric interpretation.

We will show in Lemma~\ref{lem:geom} that the surface $S$ is an
equivariant compactification of $\mathbb{G}_a^2$, so that the work
of Chambert-Loir and Tschinkel \cite{CT02} applies, where they have
already achieved an analytic continuation of the associated height
zeta function and an asymptotic formula for the counting problem.
However, our results are stronger for a number of reasons. Firstly,
we do not use the fact that $S$ is an equivariant compactification
of $\mathbb{G}_a^2$, so our methods seem applicable to more general
situations. We also get an explicit expression for the height zeta
function in terms of the Riemann zeta function, which gives a better
insight into how these zeta functions look and behave for a concrete
example. Furthermore, whereas \cite{CT02} only gives a holomorphic
continuation of $(s-1)^4Z_{U,H}(s)$ to an unspecified half-plane
$\re(s)>1-\delta$, we are able to show that $\delta=1/6$ is
acceptable, and that $\delta=1/4$ appears to be a natural boundary
under the assumption of the Riemann hypothesis. As a consequence, we
get an explicit (and stronger) error term in our asymptotic formula.

The first important step in the proof of Theorem~\ref{thm:HZF} is to
relate the counting problem on $S$ to that of counting integral
points on the associated \emph{universal torsor}. Universal torsors
were introduced by Colliot-Th\'{e}l\`{e}ne and Sansuc in
\cite{CTS87} to aid the study of the Hasse principle and weak
approximation. However, Salberger \cite{Sal98} showed that they
could be a valuable tool in counting problems on varieties. In
general a variety may have more than one universal torsor, however
in our case there is only one universal torsor (see Section
\ref{subsec:torsor} for further details). It can be visualised as a
certain open subset $\mathcal{T}$ of the affine variety in
$\mathbb{A}^7$ given by the following equation
$$\eta_2\alpha_1^2+ \eta_3\alpha_2 +\eta_4\alpha_3 = 0.$$
For our purposes, the universal torsor is a variety with a
surjective morphism $\pi:\mathcal{T} \to S$ defined over
$\mathbb{Q}$, and an action of $\mathbb{G}_m^4$ on $\mathcal{T}$
which preserves the fibres of $\pi$ and acts freely and transitively
on them. Exact definitions can be found in the above references, and
a concrete realisation of the universal torsor can be found in Lemma
\ref{lem:torsor}.

To relate the two counting problems we find a suitable set-theoretic
section of the map $\pi$, which corresponds to requiring that we
count certain integral points satisfying the universal torsor
equation and certain coprimality conditions. Previous methods for
achieving this in similar problems have been the ``elementary
method" \cite[Section 4]{BB07} and the ``blow-up method"
\cite[Section 4]{DT07}. The first method involves looking for
divisibility relations given by the equations of the surface, and
then performing a lengthy chain of substitutions to pull out any
highest common factors among the variables. The second method
involves knowing which exact points of $\mathbb{P}^2$ are blown-up
to create your surface, and using these to guide you through various
algebraic manipulations.

Here we present a new method, which uses the action of
$\mathbb{G}_m^4$ on the universal torsor. Essentially, we use this
action to ``rescale" each point in each fibre to a unique point.
Since the universal torsor (if it exists) of a more general variety
always has a free and transitive group action on its fibres, this
method is more likely to generalise to other situations than the
previously two mentioned methods. See Lemma~\ref{lem:section} for
more details.

\textbf{Notation}: To simplify notation, throughout this paper
$\varepsilon$ is any positive real number which all implied
constants are allowed to depend upon. We use the common practice
that $\varepsilon$ can take different values at different points of
the argument.

\textbf{Acknowledgments}: The author is funded by an EPSRC student
scholarship and is grateful for the help and support of Tim
Browning, and for useful conversations with Per Salberger, Emmanuel
Peyre, Ulrich Derenthal, Tomer Schlank, Tony Scholl and R\'{e}gis de
la Bret\`{e}che. We are also indebted to the referee for their
careful reading of the preliminary manuscript and many useful
comments.

\section{Preliminary Steps}
\subsection{Some Geometry}

The underlying geometry of the surface $S$ is well understood, and
we gather some facts about it in the following lemma, which also
helps to fix some notation.

\begin{lem}\label{lem:geom}
    Let $S$ be given by (\ref{equations}). Then the following holds.
    \begin{itemize}
        \item $S$ is a split singular del Pezzo surface of degree $6$ given by
        its anticanonical embedding.
        \item It contains the singular point $(1:0:0:0:0:0:0)$ of type
            $\mathbf{A}_2$.
        \item The only lines in $S$ are given by
            \begin{align*}
                &L_1: x_1=x_3=x_4=x_5=x_6=0, \\
                &L_2: x_2=x_3=x_4=x_5=x_6=0.
            \end{align*}
            In particular $U=S\setminus\{L_1 \cup
            L_2\}=S\setminus\{x_5=0\}$.
        \item $S$ is the closure of $\mathbb{P}^2$ under the rational
        map $\varphi:\mathbb{P}^2 \dashrightarrow S$ given by
        \begin{align*}
            &\varphi(x_3:x_5:x_6) =(\varphi_0(x_3,x_5,x_6):\cdots:\varphi_6(x_3,x_5,x_6))= \\
            &(-x_3^2x_5-x_3 x_6^2:x_3x_5x_6:-x_3x_5x_6-x_6^3:
            x_3x_5^2:-x_3x_5^2-x_5x_6^2:x_5^3:x_5^2x_6),
        \end{align*}
        where
        $\Gamma(\mathbb{P}^2,\mathcal{O}_{\mathbb{P}^2}(1))=\langle x_3,x_5,x_6 \rangle$.
        \item The group law on $\varphi(\mathbb{G}_a^2)= U$
        extends to an action on $S$ by translation. i.e. $S$ is
        an equivariant compactification of $\mathbb{G}_a^2$.
    \end{itemize}
\end{lem}
\begin{proof}
    First, it is clear that $\varphi$ defines
    an isomorphism $U \cong \mathbb{G}_a^2$. Hence the divisor class
    group of $S$ is generated by the $L_1$ and $L_2$, as $\Pic(\mathbb{G}_a^2)=0$.
    It is simple enough to check that the induced group
    law on $U$ extends to an action on all of $S$. However as mentioned in the introduction
    we will not use this fact in this paper, so the proof is
    omitted and can be found in \cite{DL10}.

    Resolving the singularity explicitly via blow-ups
    creates two exceptional curves $E_1$ and $E_2$ on the minimal desingularisation $\widetilde{S}$.
    The singularity is of type $\mathbf{A}_2$ and
    $\Pic(\widetilde{S}) = \langle E_1,E_2,E_3,E_4 \rangle
    \cong \mathbb{Z}^4$, where $E_3$ and $E_4$
    are the strict transforms of $L_1$ and $L_2$ respectively.
    Now, one can use the adjunction formula \cite[Ch. V, Prop. 1.5]{Har77} to show that
    $-K_{\widetilde{S}}=4E_1 + 2E_2 + 3E_3 +
    3E_4$, which proves that $K_{\widetilde{S}}^2=6$. Also,
    one can show that the pull back of the hyperplane section on $S$
    is $-K_{\widetilde{S}}$, thus proving that $S$ is a singular del
    Pezzo surface of degree $6$ given by its anticanonical
    embedding.

    Finally, we note that the $\mathbf{A}_2$ singular del Pezzo surface of degree $6$
    contains only two lines by the classification of singular del Pezzo
    surfaces \cite[Prop. 8.3]{CT88}. These are both defined over $\mathbb{Q}$,
    so the surface is indeed split.
\end{proof}

We also include the extended Dynkin diagram of $\widetilde{S}$ in
Figure \ref{Dyn:Diag}, which records the intersection behaviour of
relevant curves on $\widetilde{S}$. This can be derived from the
proof of Lemma~\ref{lem:geom}, or found in \cite[Sec. 5]{Der06}.
Here $E_1,E_2,E_3$ and $E_4$ are as in the proof of Lemma
\ref{lem:geom} and
\begin{align*}
    &A_1:  S \cap \{x_1=x_2=x_6=0\} ,
    \quad A_2: S \cap\{x_0=x_1=x_3=0\},\\
    &A_3: S\cap\{x_0=x_2=x_4=0\}.
\end{align*}
These rational curves correspond to generators of the nef cone and
will be needed in our work in section \ref{subsec:torsor}.

\begin{figure}[hbt]
\centerline{
     \xymatrix{A_2 \ar@{-}[rr] \ar@{-}[dr] \ar@{-}[dd]&& E_3 \ar@{-}[dr]   \\
     & A_1 \ar@{-}[dl] \ar@{-}[r] & E_2 \ar@{-}[r] & E_1\\
     A_3 \ar@{-}[rr] &&  E_4 \ar@{-}[ur]}}
     \caption{The extended Dynkin diagram for $\widetilde{S}$.}
    \label{Dyn:Diag}
\end{figure}
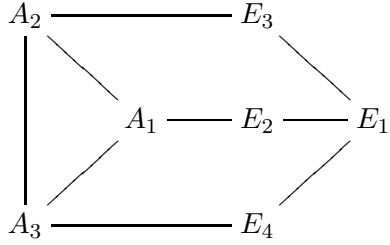

\subsection{Calculating Peyre's Constant}\label{subsec:constant}
In this section we shall verify that the constant achieved in the
asymptotic formula for Theorem~\ref{thm:asym} is in agreement with
the conjectural expression as formulated by Peyre \cite[Sec.
2]{Pey95}. Since our surface is split, it is birational to
$\mathbb{P}^2$ over $\mathbb{Q}$. So the constant is equal to the
following three factors multiplied together:
\begin{itemize}
    \item The volume $\alpha(\widetilde{S})$ of a certain polytope
    in the cone of effective divisors,
    \item The real density $\tau_\infty(\widetilde{S})$,
    \item The $p$-adic densities $\prod_{p } \tau_p(\widetilde{S})$.
\end{itemize}
By the work of \cite[Table 3]{Der07} we know that
$$\alpha(\widetilde{S}) = \frac{1}{432},$$
which is in agreement with the constant $\alpha(\widetilde{S})$ in
Theorem~\ref{thm:asym}.

We shall now calculate the real density, which corresponds to the
measure of some region, where we consider
$\widetilde{S}(\mathbb{R})$ as a real analytic manifold. Since
removing a codimension one subset does not change this volume, we
may consider the measure of the coordinate chart $U=S\setminus \{x_5
= 0\}$, with local coordinates $x_3$ and $x_6$. By Lemma
\ref{lem:geom}, this is just a reflection of the fact that our
surface is a compactification of $\mathbb{A}^2$ with $\varphi$ as a
local homeomorphism. Since $S$ is given by its anticanonical
embedding, we have by \cite[Section 2.2.1]{Pey95}

\begin{align*}
    \tau_\infty(\widetilde{S})
    & =\int_{\mathbb{R}^2}\frac{\mathrm{d}x_3 \mathrm{d}x_6}
    {\max(|x_3^2 +x_3 x_6^2|,|x_3x_6|,|x_3x_6+x_6^3|,
    |x_3|,|x_3+x_6^2|,1,|x_6|)} \\
    & =\int_{\mathbb{R}^2}
    \int_{x_5\geq \{\max(|x_3^2 +x_3 x_6^2|,|x_3x_6|,|x_3x_6+x_6^3|,
    |x_3|,|x_3+x_6^2|,1,|x_6|)}\frac{\mathrm{d}x_3 \mathrm{d}x_5 \mathrm{d}x_6}{x_5^2}
    \\
    & =3\int_{\{t,v,u \in \mathbb{R}:0<|t(ut+v^2)|,|uvt|,
    |uvt+v^3|,|u^2t|,|u^2t+uv^2|,u^3,|u^2v|\leq1\}} \mathrm{d}u\mathrm{d}v\mathrm{d}t,
\end{align*}
where we have used the change of variables
$$x_3=t/u,x_5=u^{-3},x_6=v/u.$$
Then noticing that we have the obvious automorphism $v \mapsto -v$
in the above integral, this gives the required expression for the
constant in Theorem~\ref{thm:asym}. We note that more generally, the
real density of any anticanonically embedded del Pezzo surface can
be calculated similarly by knowing which linear system of cubics in
$\mathbb{P}^2$ determines the given embedding.

The calculation of the $p$-adic densities for similar problems (see
\cite{BB07} for example) have normally involved a ``hands-on''
approach to point counting modulo $p$ for each prime $p$. Here we
opt for a more general method, which applies to any surface that is
the blow-up of $\mathbb{P}^2$ at a sequence of (possibly infinitely
near) rational points. First we recall some definitions.

\begin{df}
    Let $V$ be a non-singular projective variety defined over $\mathbb{Q}$. A \emph{model}
    for $V$ over $\mathbb{Z}$ is a projective morphism of schemes $\mathcal{V} \to
    \Spec{\mathbb{Z}}$,
    whose generic fibre is isomorphic to $V$. For each prime $p$, we denote by
    $\mathcal{V}_p=\mathcal{V}\times_{\Spec\mathbb{Z}}\Spec{\mathbb{F}_p}$ the
    reduction of $\mathcal{V}$ modulo $p$.

    We say that $V$ has \emph{everywhere good reduction} if there exists a
    model whose structure morphism is a \emph{smooth morphism}
    (i.e. $\mathcal{V}_p$ is a non-singular variety for each prime $p$).
\end{df}

\begin{lem}\label{lem:reduction}
    Let $S$ be a surface over $\mathbb{Q}$ with everywhere good reduction, and
    $\pi:\widetilde{S} \to S$ the blow-up of $S$ at a rational point
    $P$. Then $\widetilde{S}$ also has everywhere good reduction.
\end{lem}
\begin{proof}
    Let $\mathcal{S}$ be the model of $S$ with everywhere good reduction. Since $\mathcal{S}$ is projective, the rational point $P$
    extends uniquely to an integral point $\mathcal{P}$ of $\mathcal{S}$.
    Then the scheme $\widetilde{\mathcal{S}}$, which is defined to be
    the blow-up of $\mathcal{S}$ at $\mathcal{P}$, is a model for
    $\widetilde{S}$. For every prime $p$ it is clear that $\widetilde{\mathcal{S}}_p$  is simply
    the blow-up of $\mathcal{S}_p$ at a smooth $\mathbb{F}_p$-point,
    so $\widetilde{\mathcal{S}}$ also has everywhere good
    reduction.
\end{proof}

Now let $S,\mathcal{S},\widetilde{S}$ and $\widetilde{\mathcal{S}}$
be as in Lemma~\ref{lem:reduction}. Then it is clear that for every
prime $p$ we have $\#\widetilde{\mathcal{S}}_p(\mathbb{F}_p)
=\#\mathcal{S}_p(\mathbb{F}_p) + p$, since blowing up a smooth
$\mathbb{F}_p$-point replaces one $\mathbb{F}_p$-point by a copy of
$\mathbb{P}^1_{\mathbb{F}_p}$, which has $p+1$
$\mathbb{F}_p$-points. We can use this simple fact to prove the
following.

\begin{lem}\label{lem:densities}
    Let $S$ be a surface over $\mathbb{Q}$ which is the blow-up of $\mathbb{P}^2$ at
    $r$ (possibly infinitely near) rational points. Then for every
    prime $p$ the local density at $p$ is
    $$\tau_p(S)= \left(1 - \frac{1}{p}\right)^{r+1}\left(1 +
    \frac{r+1}{p} + \frac{1}{p^2}\right).$$
\end{lem}
\begin{proof}
    We begin by noting that the definition of $\tau_p(S)$ is independent of the choice
    of model, as pointed out in \cite[Def. 2.2]{Pey95}. Since $\mathbb{P}^2$ has everywhere
    good reduction, then so does $S$ by Lemma~\ref{lem:reduction}.
    Let $\mathcal{S}$ be the corresponding model, then
    $\#\mathcal{S}_p(\mathbb{F}_p)=1 + (r+1)p + p^2$ since $\#\mathbb{P}^2(\mathbb{F}_p)=1 + p + p^2$.  It is also
    clear that $\Pic(\mathcal{S}_p) \cong \mathbb{Z}^{r+1}$ with
    trivial galois action, hence the associated Artin L-function is
    $\zeta(s)^{r+1}$. This gives the correct ``convergence factors"
    and the result follows.
\end{proof}

Applying Lemma~\ref{lem:densities} to $\widetilde{S}$ (which is
split by Lemma~\ref{lem:geom}) with $r=3$, we deduce the result.

\section{The Proof}
\subsection{Passage to the Universal Torsor}\label{subsec:torsor}
As mentioned in the introduction, the first step in the proof is
transferring the problem of counting rational points on the surface
$S$, to counting integral points on the corresponding universal
torsor $\mathcal{T}$.

A variety may in general have more than one universal torsor,
however in our case there is only one. Indeed if a smooth projective
variety $V$ over a field $k$ has a universal torsor, then the set of
isomorphism classes of universal torsors is a principal homogeneous
space under $H^1(k,T)$, where $T=\Hom(\Pic(V),\mathbb{G}_m)$ is the
N\'{e}ron-Severi torus. However in our case $T=\mathbb{G}_m^4$ since
$\Pic(\widetilde{S})\cong\mathbb{Z}^4$ with trivial galois action,
and also $H^1(\mathbb{Q},\mathbb{G}_m^4)=0$ by Hilbert's theorem
$90$. Hence $\widetilde{S}$ can have at most one universal torsor.
However, the existence of a rational point on $\widetilde{S}$
implies the existence of a universal torsor. These facts (and more)
can be found in \cite[Sec. 2.3]{Sko01}. The following lemma gives us
a concrete description of the universal torsor.

\begin{lem} \label{lem:torsor}
    Let
    $$\Cox(\widetilde{S})=\bigoplus_{(n_1,n_2,n_3,n_4) \in \mathbb{Z}^4}
            H^0(\widetilde{S},\mathcal{O}(E_1)^{\otimes n_1}\otimes \cdots \otimes
            \mathcal{O}(E_4)^{\otimes n_4})$$
    be the Cox ring of $\widetilde{S}$. Then
    \begin{itemize}
    \item $\Cox(\widetilde{S}) \cong
        \mathbb{Q}[\alpha_1,\alpha_2,\alpha_3,\eta_1,\eta_2,\eta_3,\eta_4]
        /(\eta_2\alpha_1^2 + \eta_3\alpha_2 + \eta_4\alpha_3).$
    \item The universal torsor $\mathcal{T}$ of $\widetilde{S}$ is an
    open subset of $\Spec(\Cox(\widetilde{S}))$.
    \item We have a commutative diagram
    $$\xymatrix{\mathcal{T} \ar[r]^{\widetilde{\pi}} \ar[dr]_{\pi} & \widetilde{S} \ar[d] \\ & S }$$
    where $\pi$ is the map
    \begin{equation}
    \begin{split} \label{pi}
        \pi(\bfeta,\bfalpha) \mapsto &(
        \alpha_2 \alpha_3: \eta_1\eta_2\eta_3\alpha_1\alpha_2 :
        \eta_1\eta_2\eta_4\alpha_1\alpha_3:
        \eta_1^2\eta_2\eta_3^2\eta_4\alpha_2 \\
        &:\eta_1^2\eta_2\eta_3\eta_4^2\alpha_3 :
        \eta_1^4\eta_2^2\eta_3^3\eta_4^3:
        \eta_1^3\eta_2^2\eta_3^2\eta_4^2\alpha_1  ).
    \end{split}
    \end{equation}
    \item The action of a point $(k_1,k_2,k_3,k_4) \in \mathbb{G}_m^4$ on the universal torsor
    is $\eta_i \mapsto k_i \eta_i$ for $i=1,2,3,4$, and
    \begin{align*}
        &\alpha_1 \mapsto
        k_1k_3k_4\alpha_1, \quad
        \alpha_2 \mapsto k_1^2k_2k_3k_4^2\alpha_2,\quad
         \alpha_3 \mapsto k_1^2k_2k_3^2k_4\alpha_3.
    \end{align*}
    \end{itemize}
\end{lem}
\begin{proof}
    The calculation of the Cox ring, the map $\pi$ and the action of the
    N\'{e}ron-Severi torus on the Cox ring can be found in
    \cite{Der06}. That the universal torsor is an open subset of
    $\Spec(\Cox(\widetilde{S}))$ is well-known, see \cite[Cor. 2.16,
    Prop. 2.9]{HK00} for example.
\end{proof}

In fact, everything we need to know about the universal torsor can
be deduced from first principles. Firstly, it is not actually
necessary for us to calculate explicitly which open subset of
$\Spec(\Cox(\widetilde{S}))$ the universal torsor corresponds to.
However, it is easy to check that the action given in Lemma
\ref{lem:torsor} is well-defined  and that it preserves the fibres
of $\pi$. Moreover, $\pi$ is surjective on its domain of definition
since $\varphi^{-1} \circ\pi$ is surjective onto $U$, where
$\varphi$ and $U$ are as in Lemma~\ref{lem:geom}. And also, it is
easy enough to see that $\pi$ hits every point on $S\setminus U$ as
well, hence it is surjective.

In particular, when we consider the universal torsor as being over
$U$, it is simple to see that we get a free and transitive action on
the fibres of $\pi$ on the corresponding open subset where
$\eta_1\eta_2\eta_3\eta_4\neq0$. That is, it is clear that
$\Spec(\Cox(\widetilde{S}))\setminus\{\eta_1\eta_2\eta_3\eta_4=0\}$
is a $U$-torsor under $\mathbb{G}_m^4$.

Now to find a suitable section of the morphism $\pi$. Bearing in
mind that we are counting points on $U$ where $x_5 \neq 0$, we see
that the rational points which have some coordinate equal to zero
lie in the image of the points on the torsor where
$\alpha_1\alpha_2\alpha_3=0$. These are exactly the curves $A_1,A_2$
and $A_3$ given in Figure \ref{Dyn:Diag}. They are rational curves,
and it is easy enough to show that the corresponding counting
functions satisfy
$$N_{A_1}(B)=\frac{12}{\pi^2}B + O(B^{1/2}), \quad
N_{A_2}(B)=N_{A_3}(B)=O(B^{2/3}).$$

Since these have been taken into account, we can now assume that
each coordinate of each rational point is non-zero.

\begin{lem} \label{lem:section}
Above each rational point $x \in U(\mathbb{Q})$ with non-zero
coordinates, there is a unique integral point $(\bfalpha,\bfeta)$ on
the universal torsor satisfying
\begin{align*}
        &(\alpha_1,\eta_1\eta_3\eta_4)=(\alpha_2,\eta_1\eta_2\eta_4)=(\alpha_3,\eta_1\eta_2\eta_3)=1, \\
        &(\eta_2,\eta_3)=(\eta_2,\eta_4)=(\eta_3,\eta_4)=1, \\
        &\eta_1,\eta_2,\eta_3,\eta_4 > 0, \alpha_1\alpha_2\alpha_3
        \neq0.
\end{align*}
\end{lem}
\begin{proof}
    We should note that we are guided to the above coprimality conditions
    by Figure \ref{Dyn:Diag}, whereby two
    variables are coprime if and only if the corresponding curves do
    not intersect each other.

    First let $(\bfalpha,\bfeta)$ be an integral point on the
    universal torsor lying above a rational point with non-zero
    coordinates. Suppose that there is a prime $p \mid(\eta_1,\alpha_1)$. Then
    using the torsor action in Lemma~\ref{lem:torsor} with $k_1=1/p,k_2=p^3,k_3=k_4=1$, we
    map
    \begin{align*}
        \eta_1 & \mapsto \eta_1/p, \quad \eta_2 \mapsto p^3\eta_2, \\
        \alpha_1 &\mapsto \alpha_1/p, \quad\alpha_2 \mapsto p\alpha_2,
         \quad\alpha_3 \mapsto p\alpha_3.
    \end{align*}
    So we have successfully managed to divide $\eta_1$ and
    $\alpha_1$ by $p$, and left the other variables as integers, meaning that if they have any common factor
    we can remove it. A very similar argument
    works for $\eta_3$ and $\eta_4$, so we can assume
    $$(\alpha_1,\eta_1\eta_3\eta_4)=1.$$
    We now fix our choice of $\alpha_1$ modulo $\{\pm1\}$, meaning that from now on we
    impose the condition $|k_1k_3k_4|=1$. This simplifies the action
    on $\alpha_2$ and $\alpha_3$ to
    $$\alpha_2 \mapsto \frac{k_2}{k_3}\alpha_2, \quad \alpha_3 \mapsto
    \frac{k_2}{k_4}\alpha_3.$$
    Carrying on with the same procedure, if $p \mid
    (\alpha_2,\eta_1)$, take $k_1=1/p,k_2=k_4=1,k_3=p$ to get
    $(\alpha_2,\eta_1)=1$ and for $p \mid (\alpha_2,\eta_4)$ take
    $k_1=k_2=1, k_3=p,k_4=1/p$ to get $(\alpha_2,\eta_4)=1$.

    We have now come to interesting part of the proof, since so far we have not
    used the equation of the universal torsor, but now we are driven to
    use it since it encodes divisibility conditions. Namely, if $p \mid
    (\alpha_2,\eta_2)$, then $p$ must also divide $\eta_4$ or $\alpha_3$. But
    $(\alpha_2,\eta_4) =1$, so we are safe to choose
    $k_1=k_3=k_4=1,k_2=1/p$ and keep $\alpha_3$ as an integer. So
    we have successfully shown that we can choose
    $$(\alpha_2,\eta_1\eta_2\eta_4)=1.$$
    We fix this choice of $\alpha_2$ modulo $\{\pm1\}$, which is equivalent to
    requiring $|k_2|=|k_3|$.

    The reader should now be familiar with the method and can check
    that we can assume $(\alpha_3,\eta_1\eta_2\eta_3)=1$ after
    performing the following
    \begin{itemize}
        \item If $p \mid (\alpha_3,\eta_1)$, choose
        $k_1=1/p,k_2=k_3=1,k_4=p$,
        \item If $p \mid (\alpha_3,\eta_3)$, choose
        $k_1=p,k_2=k_3=1/p,k_4=1$,
        \item If $p \mid (\alpha_3,\eta_2)$, contradiction since
        $(\alpha_3,\eta_3)=(\alpha_2,\eta_2)=1$.
    \end{itemize}
    So fixing $\alpha_3$ modulo $\{\pm1\}$, we are restricted to
    $$|k_2|=|k_3|=|k_4|.$$
    But if $p \mid (\eta_2,\eta_3,\eta_4)$, choosing
    $k_1=p^2,k_2=k_3=k_4=1/p$ then gives $(\eta_2,\eta_3,\eta_4)=1$, and moreover
    the torsor equation implies they must also be pairwise coprime.
    Finally, by choosing the $\eta_i$ to be positive, we have used all
    degrees of freedom in the torsor action and so the choice of
    integral point is unique.
\end{proof}

Using this lemma, we see that counting those points $x \in
U(\mathbb{Q})$ satisfying the height bound $H(x) \leq B$, is
equivalent to counting the unique integral points  above them on the
universal torsor which satisfy the bound
$H(\pi(\bfalpha,\bfeta))\leq B$. Naively, this corresponds to $7$
separate height conditions. However, using the map $\varphi$ from
Lemma~\ref{lem:geom}, we know that we actually have 3 degrees of
freedom. With this in mind, we define

\begin{align} \nonumber
    X_3 &= \left(\frac{\eta_1^2\eta_2\eta_3^2\eta_4}{X_5^2B}\right) =
    \left(B\eta_1^2\eta_2\eta_4^3\right)^{-1/3}, \nonumber\\
    X_5 &= \left(\frac{\eta_1^4\eta_2^2\eta_3^3\eta_4^3}{B}\right)^{1/3},\label{height1}\\
    X_6 &= \left(\frac{\eta_1^3\eta_2^2\eta_3^2\eta_4^2}{X_5^2B}\right) =
    \left(\frac{\eta_1\eta_2^2}{B}\right)^{1/3}, \nonumber\\ \nonumber
\end{align}
and let $\overline{\varphi_i}(\alpha_1,\alpha_2) =\varphi_i(\alpha_2
X_3,X_5,\alpha_1 X_6)$ for $i=0,1,2,4,$ and
$\overline{\varphi_3}(\alpha_2) =\varphi_3(\alpha_2 X_3,X_5,1)
,\overline{\varphi_6}(\alpha_1) =\varphi_6(1,X_5,\alpha_1 X_6)$.
Then it is clear that the height condition
$H(\pi(\bfalpha,\bfeta))\leq B$ is equivalent to the condition
\begin{align}
    &|\overline{\varphi_i}(\alpha_1,\alpha_2)|,|\overline{\varphi_3}(\alpha_2)|\leq 1, i =0,1,2,4, \label{height2}\\
    &X_5,\overline{\varphi_6}(\alpha_1) \leq 1 \label{height3}.
\end{align}

Finally, on noticing we have the obvious automorphism $\alpha_1
\mapsto -\alpha_1$ on the torsor, we have shown the following.
\begin{lem} \label{problem1}
    The counting function for $U$ satisfies
    $$N_U(B) = 2T(B) + \frac{12}{\pi^2}B + O(B^{2/3})$$
    where
    $$ T(B) = \#\left\{
    \begin{array}{ll}
        (\bfalpha,\bfeta) \in \mathbb{Z}^7
    \end{array}:
    \begin{array}{ll}
        \eta_2\alpha_1^2+ \eta_3\alpha_2 +\eta_4\alpha_3 = 0, (\ref{height2}),(\ref{height3}), \\
        (\alpha_1,\eta_1\eta_3\eta_4)=(\alpha_2,\eta_1\eta_2\eta_4)= (\eta_2,\eta_3)=1,\\
        (\alpha_3,\eta_1\eta_2\eta_3)=(\eta_2,\eta_4)=(\eta_3,\eta_4)=1, \\
        \alpha_1,\eta_1,\eta_2,\eta_3,\eta_4 > 0, \alpha_2\alpha_3
        \neq0.
    \end{array} \right\} .$$
\end{lem}

We note that we have the natural upper bound $\alpha_1 \leq 1/X_5^2
X_6$ given by $\overline{\varphi_6}$. However, we can actually do
better than this, which will be quite important to improving our
error term later on. Notice that $\overline{\varphi_4}$ and
$\overline{\varphi_3}$ imply
$$ -\frac{1}{X_3X_5^2} \leq  \alpha_2 \leq \frac{1}{X_3X_5^2}\left(1 -
    \alpha_1^2X_5X_6^2\right).$$
Rearranging this in terms of $\alpha_1$, we deduce the stronger
bound
\begin{equation}
    \alpha_1\leq \frac{\sqrt{2}}{X_6\sqrt{X_5}}. \label{bound:alpha1}
\end{equation}

\subsection{\Mob Inversion}
Now we shall use \Mob inversion to remove the coprimality conditions
on the $\alpha_i$'s. Recalling the counting problem in Lemma
\ref{problem1} and the height conditions (\ref{height3}), it makes
sense to define
\begin{align}
    \mathcal{N}=&
    \left\{\bfeta \in \mathbb{Z}^4 :
    \begin{array}{ll}
        \eta_1,\eta_2,\eta_3,\eta_4 > 0, X_5 \leq 1,\\
        (\eta_2,\eta_3)=(\eta_2,\eta_4)=(\eta_3,\eta_4)=1.
    \end{array}
    \right\} .\label{bigeta}
\end{align}
Then it is clear that
$$T(B)=\sum_{\bfeta \in \mathcal{N}} \sum_{\substack{\alpha_1 > 0 \\
(\alpha_1,\eta_1\eta_3\eta_4)=1 \\ \overline{\varphi_6}(\alpha_1)
\leq 1}} S $$ where
$$S = \#\left\{\alpha_2,\alpha_3 \in \mathbb{Z} :
    \begin{array}{ll}
        \alpha_2\alpha_3 \neq 0, (\ref{height2}) \mbox{ holds},\\
        (\alpha_2,\eta_1\eta_2\eta_4)=(\alpha_3,\eta_1\eta_2\eta_3)=1, \\
        \eta_2\alpha_1^2+ \eta_3\alpha_2 +\eta_4\alpha_3 = 0.
    \end{array}
    \right\}.$$
Now using \Mob inversion on $(\alpha_3,\eta_1\eta_2\eta_3)=1$ gives
us
$$S = \sum_{k_3\mid\eta_1\eta_2\eta_3}\mu(k_3) S_{k_3}$$
where

$$S_{k_3} = \#\left\{\alpha_2,\alpha_3 \in \mathbb{Z} :
    \begin{array}{ll}
        \alpha_2\alpha_3 \neq 0,(\ref{height2}) \mbox{ holds}, \\
        (\alpha_2,\eta_1\eta_2\eta_4)=1, \\
        \eta_2\alpha_1^2+ \eta_3\alpha_2 +k_3\eta_4\alpha_3 = 0.
    \end{array}
    \right\}.$$
However $S_{k_3} \neq 0$ if and only if $(k_3,\eta_2\eta_3)=1$, so
$$S = \sum_{\substack{k_3\mid\eta_1 \\(k_3,\eta_2\eta_3)=1}}\mu(k_3) S_{k_3}.$$
A similar argument yields
\begin{equation}
    T(B)=\sum_{\bfeta \in \mathcal{N}}
        \sum_{\substack{k_3\mid\eta_1 \\ (k_3,\eta_2\eta_3)=1}}\mu(k_3)
        \sum_{\substack{k_2\mid\eta_1\eta_2 \\ (k_2,k_3\eta_4)=1}} \mu(k_2)
        \sum_{\substack{\alpha_1 > 0 \\ (\alpha_1,\eta_1\eta_3\eta_4)=1 \\ \overline{\varphi_6}(\alpha_1) \leq 1}}
        S_{k_2,k_3} \label{problem2}
\end{equation}
where $$S_{k_2,k_3} = \#\left\{\alpha_2,\alpha_3 \in \mathbb{Z} :
    \begin{array}{ll}
        \alpha_2\alpha_3 \neq 0, k_2\eta_3\alpha_2 + \eta_2\alpha_1^2 +k_3\eta_4\alpha_3=0,\\
        |\overline{\varphi_i}(\alpha_1,k_2\alpha_2)|,|\overline{\varphi_4}(k_2\alpha_2)|\leq 1, i
        =0,1,2,3.
    \end{array}
    \right\}.$$

\subsection{Sum over $\alpha_2$ and $\alpha_3$ via Congruences}
In this section, we shall perform the summation over $\alpha_2$. We
note that there are no conditions on $\alpha_3$ other than the
equation of the universal torsor, so we find that

$$S_{k_2,k_3} = \#\left\{\alpha_2 \in \mathbb{Z} :
    \begin{array}{ll}
        \alpha_2 \neq 0, k_2\eta_3\alpha_2 \equiv -\eta_2\alpha_1^2 \pmod{k_3\eta_4},\\
        |\overline{\varphi_i}(\alpha_1,k_2\alpha_2)|,|\overline{\varphi_4}(k_2\alpha_2)|\leq 1, i
        =0,1,2,3.
    \end{array}
    \right\}.$$
However since $(k_2\eta_3,k_3\eta_4)=1$, $\alpha_2$ is uniquely
determined modulo $k_3\eta_4$. For any integers $q,n_0,a,b$ with
$a<b$, we have the simple estimate
$$\#\{n \in \mathbb{Z} \cap [a,b]: n \equiv n_0 \pmod q\} =
    \frac{b-a}{q} + O(1).$$
Using this and the change of variables $t \mapsto k_2tX_3$, we see
that
\begin{equation}
    S_{k_2,k_3} = \frac{1}{k_2k_3\eta_4X_3}F_1(X_5,\alpha_1X_6)+
    O(1) \label{def:Sk2k3}
\end{equation}
where $F_1(u,v)$ is defined by the following result.

\begin{lem}\label{lem:F1}
    Let $$F_1(u,v)
    =\int_{\{t\in
    \mathbb{R}:,0<|t(ut+v^2)|,|uvt|,|uvt+v^3|,|u^2t|,|u^2t+uv^2|\leq1\}}
    \mathrm{d}t$$ for $u,v\geq0$ and $(u,v) \neq (0,0)$. Then
    \begin{enumerate}
        \item[(a)] For $u \neq 0$, $$F_1(u,v) \leq \frac{2}{\sqrt{u}}.$$
        \item[(b)] $F_1(u,v)$ is piecewise differentiable with respect to $u$ and $v$.
    \end{enumerate}
\end{lem}
\begin{proof}
    The differentiability condition is clear, so it remains to prove the inequality.
    First let $M(u,v)=\vol\{t\in \mathbb{R}:|t(ut+v^2)|
    \leq 1\}$, then we have
    $M(u,v) = \vol\{t\in \mathbb{R}: v^4/4u^2-1/u \leq t^2 \leq 1/u + v^4/4u^2\}$
    after completing the square. If $v^4/4u^2 \geq 1/u$, then using
    the simple fact that $\sqrt{a+b} \leq \sqrt{a} + \sqrt{b}$ for
    all non-negative real numbers $a$ and $b$, we deduce that
        $$M(u,v) = \sqrt{v^4/4u^2 + 1/u} - \sqrt{v^4/4u^2- 1/u} \leq
        \sqrt{2/u}.$$
    Similarly, if $v^4/4u^2 \leq 1/u$ then $M(u,v) = \sqrt{v^4/4u^2 +
    1/u} \leq 2/\sqrt{u}$.
\end{proof}

We now have our first error term in the counting problem
(\ref{problem2}). First recall  that $\sum_{k\mid n} |\mu(k)| =
2^{\omega(n)}$ where $\omega(n)$ is the number of prime divisors of
$n$, that we have the stronger bound on $\alpha_1$ given by
(\ref{bound:alpha1}), and the definition (\ref{bigeta}) of
$\mathcal{N}$. Using these, we see that the overall contribution to
the error term from (\ref{def:Sk2k3}) is

\begin{align*}
    & \ll \sum_{\eta_1^4\eta_2^2\eta_3^3\eta_4^3\leq B}
    \sum_{k_3\mid\eta_1} |\mu(k_3)|
    \sum_{k_2\mid\eta_1\eta_2} |\mu(k_2)|
    \sum_{|\alpha_1|\leq \frac{\sqrt{2}}{X_6\sqrt{X_5}}} 1 \\
    & \ll B^{1/2} \sum_{\eta_1^4\eta_2^2\eta_3^3\eta_4^3\leq B}
    \frac{2^{\omega(\eta_1)} 2^{\omega(\eta_1\eta_2)}
    }{\eta_1\eta_2\eta_3^{1/2}\eta_4^{1/2}} \\
    & \ll B^{1/2} \sum_{\eta_1^4\eta_2^2\eta_3^3\leq B}
    \frac{4^{\omega(\eta_1)} 2^{\omega(\eta_2)}
    }{\eta_1\eta_2\eta_3^{1/2}} \cdot\frac{B^{1/6}}{\eta_1^{2/3}\eta_2^{1/3}\eta_3^{1/2}}
    \ll B^{2/3+\varepsilon}
\end{align*}
since $2^{\omega(n)} \leq d(n) \ll n^\varepsilon$, where $d(n)$ is
the usual divisor function. This error term is clearly satisfactory
for Theorem~\ref{thm:asym}.

\subsection{Sum over $\alpha_1$}
Recall that the main term in our counting problem is given by
(\ref{problem2}) and (\ref{def:Sk2k3}). Applying \Mob inversion to
remove the coprimality condition in the sum over $\alpha_1$ gives
$$\sum_{\substack{\alpha_1 > 0 \\ (\alpha_1,\eta_1\eta_3\eta_4)=1 \\ \overline{\varphi_6}(\alpha_1) \leq 1}}
        F_1(X_5,\alpha_1X_6) =
    \sum_{k_1\mid\eta_1\eta_3\eta_4}\mu(k_1) \sum_{0<\alpha_1 \leq 1/k_1X_5^2X_6} F_1(X_5,\alpha_1k_1X_6).$$
A natural step is to now apply Euler-Maclaurin summation. To
simplify our notation in what follows, we shall use Stieltjes
integral notation, and also use $\{\cdot\}$ to denote the fractional
part of a real number.

\begin{lem}\label{lem:F2E}
    We have
    $$\sum_{0<\alpha_1 \leq 1/k_1X_5^2X_6} F_1(X_5,\alpha_1k_1X_6) =
    \frac{1}{k_1X_6}F_2(X_5) + E(\bfeta,k_1,B),$$
    where for $u>0$ we have
    \begin{align*}
    F_2(u)&=\int_0^{\frac{1}{u^2}}F_1\left(u,v\right)\mathrm{d}v ,\\
    &=\int_{\{t,v \in \mathbb{R}:0<|t(ut+v^2)|,|uvt|,|uvt+v^3|,|u^2t|,|u^2t+uv^2|,u^2v\leq1\}}
    \mathrm{d}v\mathrm{d}t ,
    \end{align*}
    and
    $$E(\bfeta,k_1,B) =\int_0^{1}\left\{\frac{v}{k_1X_5^2X_6}\right\}\mathrm{d}F_1\left(X_5,\frac{v}{X_5^2}\right)
    -\left\{\frac{1}{k_1X_5^2X_6}\right\}F_1\left(X_5,\frac{1}{X_5^2}\right).$$
    We also have the bounds
    \begin{equation}
        |E(\bfeta,k_1,B)| \leq \frac{6}{\sqrt{X_5}}, \qquad F_2(u) \leq
        \frac{4}{\sqrt{u}}. \label{bounds:F2E}
    \end{equation}
\end{lem}
\begin{proof}
    Euler-Maclaurin summation gives
    \begin{align*}
    &\sum_{0<\alpha_1 \leq 1/k_1X_5^2X_6} F_1(X_5,\alpha_1k_1X_6) \\
    & =\int_0^{\frac{1}{k_1X_5^2X_6}}F_1\left(X_5,vk_1X_6\right)\mathrm{d}v -
    \int_0^{\frac{1}{k_1X_5^2X_6}}F_1\left(X_5,vk_1X_6\right)\mathrm{d}\{v\}.
    \end{align*}
    Changing variables and applying integration by parts gives
    the first part of the lemma. As for the first upper bound, recall the properties of $F_1$ given in
    Lemma~\ref{lem:F1}. Then we have
    $$ \left|E(\bfeta,k_1,B)\right|
        \leq  2F_1\left(X_5,\frac{1}{X_5^2}\right) + F_1\left(X_5,0\right)
        \leq \frac{6}{\sqrt{X_5}}. $$
    For the second upper bound, note that $|uvt|\leq 1$ and $|uvt +
    v^3|\leq 1$ imply that $v \leq 2^{1/3}$, hence
    $$F_2(u) \leq \int_0^{2^{1/3}}F_1\left(u,v\right)\mathrm{d}v \leq
    \frac{4}{\sqrt{u}}.$$
\end{proof}

\subsection{Making a Lower Order Term Explicit}\label{subsec:error}
The counting problem (\ref{problem2}) now stands as
\begin{align*}
    T(B)=& \sum_{\bfeta \in \mathcal{N}}
        \frac{F_2(X_5)}{\eta_4X_3X_6}
        \sum_{\substack{k_3\mid\eta_1 \\
        (k_3,\eta_2\eta_3)=1}}\frac{\mu(k_3)}{k_3}
        \sum_{\substack{k_2\mid\eta_1\eta_2 \\ (k_2,k_3\eta_4)=1}}
        \frac{\mu(k_2)}{k_2}
        \sum_{k_1\mid\eta_1\eta_3\eta_4}
        \frac{\mu(k_1)}{k_1}  \\
        &+T_1(B)
\end{align*}
where $T_1(B)$ denotes the same expression, but with $
F_2(X_5)/k_1X_6$ replaced by $E(\bfeta,k_1,B)$. It turns out that
there is a term of order $B$ in $T_1$, which we shall handle by
performing the sum over $\eta_2$ explicitly. Taking out the factors
which depend on $\eta_2$ and recalling the definition of
$\mathcal{N}$ in (\ref{bigeta}) and the height conditions
(\ref{height1}), we see that
\begin{equation}
    T_1(B)=B^{1/3}\sum_{\substack{\eta_1^4\eta_3^3\eta_4^3 \leq B \\
    (\eta_3,\eta_4)=1 }}
        \eta_1^{2/3} \sum_{k_1\mid\eta_1\eta_3\eta_4}\mu(k_1)
        T_2\left(\eta_1,\eta_3,\eta_4,k_1, \widetilde{X_5}\right)
        \label{def:T1}
\end{equation}
where we define
\begin{equation}
    \widetilde{X_5}=\eta_2/X_5^{3/2}=\sqrt{B/(\eta_1^4\eta_3^3\eta_4^3)}
    \label{X5:tilde}
\end{equation}
and
\begin{align*}
    T_2(\eta_1,\eta_3,\eta_4,k_1,\widetilde{X_5})&=  \\
    \sum_{\substack{\eta_2 \leq \widetilde{X_5} \\
        (\eta_2,\eta_3\eta_4)=1}} &\eta_2^{1/3}E(\bfeta,k_1,B)
        \sum_{\substack{k_3\mid\eta_1 \\
        (k_3,\eta_2\eta_3)=1}}\frac{\mu(k_3)}{k_3}
        \sum_{\substack{k_2\mid\eta_1\eta_2 \\ (k_2,k_3\eta_4)=1}}
        \frac{\mu(k_2)}{k_2} .
\end{align*}

This is essentially a sum involving an arithmetic function and a
real valued function, so partial summation is the natural method to
use. However first we need to unravel this arithmetic function to
get a multiplicative function in $\eta_2$. To simplify our notation,
let
\begin{equation}
    \phi^*(a_1,\ldots,a_n)=\prod_{p\mid(a_1,\ldots,a_n)}\left(1 -
    \frac{1}{p}\right), \label{phi^*}
\end{equation}
and we use the shorthand $\phi^*(a)=\phi^*(a,a)$. There will also be
unfortunate $2$-adic conditions we shall need to take care of, so we
define
$$
\mathcal{N}_0 = \{(\eta_1,\eta_3,\eta_4) \in \mathbb{N}^3:
    2\nmid\eta_1 \mbox{ or } 2 \mid \eta_3\eta_4\},\quad
\mathcal{N}_1 =  \mathbb{N}^3 \setminus \mathcal{N}_0.
$$

\begin{lem} \label{lem:nu}
We have
$$T_2(\eta_1,\eta_3,\eta_4,k_1,\widetilde{X_5})= \psi(\eta_1,\eta_3,\eta_4)
        \sum_{\eta_2 \leq \widetilde{X_5}} \nu_{\eta_1,\eta_3,\eta_4}(\eta_2) \eta_2^{1/3}E(\bfeta,k_1,B),$$
where
\begin{align*}
    \psi(\eta_1,\eta_3,\eta_4)& = \phi^*(\eta_1,\eta_3\eta_4)
    \prod_{\substack{p\mid \eta_1 \\ p \nmid \eta_3\eta_4 \\ p\neq2}}
    \left(1-\frac{2}{p}\right), \\
    \widetilde{\nu}_{\eta_1,\eta_3,\eta_4}(\eta_2)& =
    \left\{
    \begin{array}{ll}
        \phi^*(\eta_2) \prod_{\substack{p\mid \eta_1,\eta_2 \\ p \neq 2}} \left(1-\frac{2}{p}\right)^{-1}
        ,&(\eta_2,\eta_3\eta_4)=1,\\
        0,& \mbox{otherwise},
    \end{array}
    \right.
\end{align*}
and if $(\eta_1,\eta_3,\eta_4) \in \mathcal{N}_i$, then
$$\nu_{\eta_1,\eta_3,\eta_4}(\eta_2)=\left\{
\begin{array}{ll}
        \widetilde{\nu}_{\eta_1,\eta_3,\eta_4}(\eta_2), &2^{i}\mid\eta_2,\\
        0,& \mbox{otherwise}.
    \end{array}\right.$$
\end{lem}
\begin{proof}
    One can verify the following expression
    \begin{align*}
    \sum_{\substack{k_3\mid\eta_1 \\
    (k_3,\eta_2\eta_3)=1}}\frac{\mu(k_3)}{k_3}
    \sum_{\substack{k_2\mid\eta_1\eta_2 \\ (k_2,k_3\eta_4)=1}}
    \frac{\mu(k_2)}{k_2}
    =\phi^*(\eta_1,\eta_3\eta_4)\phi^*(\eta_2)
    \prod_{\substack{p\mid \eta_1 \\ p \nmid \eta_2\eta_3\eta_4 }}
    \left(1-\frac{2}{p}\right),
    \end{align*}
    by checking its value at prime powers and recalling that
    $(\eta_2,\eta_3)=(\eta_2,\eta_4)=(\eta_3,\eta_4)=1$.

    We want this to be written as a
    multiplication function of $\eta_2$ times some other arithmetic function independent of $\eta_2$. In
    order to do this, we need to split up the product over primes, but we
    can only safely do this if it is non-zero, i.e. if $2\nmid \eta_1$ or
    $2\mid\eta_2\eta_3\eta_4$. So we have defined $\nu_{\eta_1,\eta_3,\eta_4}$ be zero exactly when
    $2\mid\eta_1,2\nmid \eta_2\eta_3\eta_4$ and the coprimality conditions are not satisfied,
    and simplified its definition in the remaining cases.
\end{proof}

Note that $\widetilde{\nu}_{\eta_1,\eta_3,\eta_4}$ is a
multiplicative function of $\eta_2$, but
$\nu_{\eta_1,\eta_3,\eta_4}$ is not. The next natural step is to
find the average order of $\widetilde{\nu}_{\eta_1,\eta_3,\eta_4}$.
However to simplify our notation and argument, from now on we shall
assume that $(\eta_1,\eta_3,\eta_4) \in \mathcal{N}_0.$ The other
case is almost exactly the same, the only difference being the
condition that $\eta_2$ must be even, and it will still contribute a
power of $B$ to the main term and give the same error term. With
this in mind, we have the following.

\begin{lem}\label{lem:nu_2}
    Let $V(s)$ be the Dirichlet series associated to $\widetilde{\nu}_{\eta_1,\eta_3,\eta_4}$ and
    $\widetilde{V}(s)=V(s)/\zeta(s)$. Then $\widetilde{V}(s)$ is
    a holomorphic and bounded function on $\re(s) > 0$ satisfying $0 \leq \widetilde{V}(1)\ll
    2^{\omega(\eta_1)}$ and
    $$\sum_{n \leq X} \widetilde{\nu}_{\eta_1,\eta_3,\eta_4}(n) = \widetilde{V}(1)X + O(2^{\omega(\eta_1)}X^\varepsilon).$$
\end{lem}
\begin{proof}
    For $p\neq2$, it is easy to see that
    \begin{align*}
    \widetilde{\nu}_{\eta_1,\eta_3,\eta_4}(p^k)& = \left\{
    \begin{array}{ll}
        \left(1- \frac{1}{p}\right) , &\quad p\nmid \eta_1\eta_3\eta_4, \\
        \left(\frac{1- 1/p}{1-2/p}\right) , &\quad p \mid
        \eta_1, p \nmid \eta_3\eta_4, \\
        0 , &\quad \mbox{otherwise}.
    \end{array}\right.
    \end{align*}
    Then by considering Euler products, one can check that $V(s)$ is
    equal to
    \begin{align*}
    \frac{\zeta(s)V'(s)}{\zeta(s+1)}
    \prod_{ p \mid \eta_1\eta_3\eta_4}
    \left(1 + \frac{1-1/p}{p^s-1}\right)^{-1}
    \prod_{\substack{p \mid \eta_1 , p \neq 2\\ p \nmid \eta_3\eta_4}}
    \left(1 + \frac{1-1/p}{(1-2/p)(p^s-1)}\right)
    \end{align*}
    where $V'(s)$ is some function corresponding to the Euler factor at the prime $2$.
    So $\widetilde{V}(s)$ has the properties stated in the lemma.
    Ignoring convergence issues for now, we have
    \begin{align*}
    \sum_{n \leq X} \widetilde{\nu}_{\eta_1,\eta_3,\eta_4}(n) & =  \sum_{n \leq X}
    ((\widetilde{\nu}_{\eta_1,\eta_3,\eta_4}*\mu)*1)(n)\\
    &=  \sum_{n \leq X} \sum_{d\mid n}(\widetilde{\nu}_{\eta_1,\eta_3,\eta_4}*\mu)(d) \\
    & =  X\sum_{d =1}^\infty \frac{(\widetilde{\nu}_{\eta_1,\eta_3,\eta_4}*\mu)(d)}{d}
         +O\left(X^{\varepsilon}\sum_{d =1}^\infty \frac{|(\widetilde{\nu}_{\eta_1,\eta_3,\eta_4}*\mu)(d)|}{d^\varepsilon}\right) \\
    \end{align*}
    where we have used the trivial bound $[x]=x + O(x^\varepsilon)$. To
    make this rigorous, first note that
    $$\lim_{s\to1}\sum_{d =1}^\infty \frac{(\widetilde{\nu}_{\eta_1,\eta_3,\eta_4}*\mu)(d)}{d^s}=
    \lim_{s\to1}V(s)\zeta(s)^{-1}=\widetilde{V}(1).$$

    Next, we need to find an expression for the Dirichlet series
    $V^+(s)$ of $|(\widetilde{\nu}_{\eta_1,\eta_3,\eta_4}*\mu)|$. It is easy to verify that

    $$(\widetilde{\nu}_{\eta_1,\eta_3,\eta_4}*\mu)(p^k) = \left\{
        \begin{array}{ll}
            \widetilde{\nu}_{\eta_1,\eta_3,\eta_4}(p) - 1, &\quad k=1,\\
            0,&\quad k>1.
        \end{array}\right.$$
    By considering Euler products, one can check that $V^+(s)$ is a holomorphic and
    bounded function of $s$ on $\re (s)>0$ and satisfies
    \begin{align*}
    \widetilde{V}^+(\varepsilon)
        & \ll
        \prod_{\substack{p \mid \eta_1 \\ p \nmid \eta_3\eta_4}}\left(1 +
        \frac{1}{(p-2)p^\varepsilon}\right) \ll 2^{\omega(\eta_1)}
    \end{align*}
    on this domain. Thus we are done.
\end{proof}

We shall now perform the summation over $\eta_2$, and to do this we
will need a slight abuse of notation. Namely, we define
$$\widetilde{E}(t) = E(\eta_1,t,\eta_3,\eta_4,k_1,B),$$ where $E$ is
given in Lemma~\ref{lem:F2E}, and for this we also need to think of
$X_5$ and $X_6$ as being functions of $\eta_2$. Recalling the
expression we had for $T_2$ as given in Lemma~\ref{lem:nu} and using
Lemma~\ref{lem:nu_2},  by partial summation we have
\begin{align*}
    &\sum_{ \eta_2 \leq \widetilde{X_5}} \nu_{\eta_1,\eta_3,\eta_4}(\eta_2) \eta_2^{1/3}\widetilde{E}(\eta_2) \\
    & = \widetilde{X_5}^{1/3}\widetilde{E}(\widetilde{X_5})
    \sum_{\eta_2 \leq \widetilde{X_5}} \nu_{\eta_1,\eta_3,\eta_4}(\eta_2)
    - \int_0^{\widetilde{X_5}} \sum_{\eta_2 \leq t} \nu_{\eta_1,\eta_3,\eta_4}(\eta_2) \mathrm{d}\left(t^{1/3}\widetilde{E}(t)\right)\\
    & = \widetilde{V}(1)\int_0^{\widetilde{X_5}}t^{1/3}\widetilde{E}(t)\mathrm{d}t +
    O\left(2^{\omega(\eta_1)} |\widetilde{E}(\widetilde{X_5})|
    \widetilde{X_5}^{1/3+\varepsilon}\right) \\
    &=\widetilde{V}(1)\widetilde{X_5}^{4/3}\int_0^1u^{1/3}\widetilde{E}(u\widetilde{X_5})\mathrm{d}u
    + O\left(B^\varepsilon |\widetilde{E}(\widetilde{X_5})|
    \widetilde{X_5}^{1/3+\varepsilon}\right).
\end{align*}

We now note an interesting feature, namely that
$\widetilde{E}(\widetilde{X_5})$ is actually independent of $B$.
Indeed, viewing $X_5$ and $X_6$ as functions of $\eta_2$, we find
that $X_5(u\widetilde{X_5})=u^{2/3}$ and
$X_6(u\widetilde{X_5})=\eta_1\eta_3\eta_4/u^{2/3}$. Hence
$\widetilde{E}(u\widetilde{X_5})$ is independent of $B$ and moreover
by (\ref{bounds:F2E}) we deduce that
$$ \widetilde{E}(u\widetilde{X_5}) \leq \frac{6}{u^{1/3}}.$$
Hence referring back to (\ref{def:T1}), the overall error term
contribution to $T_1(B)$ in this case is
\begin{align*}
    &\ll B^{1/3+\varepsilon}\sum_{\eta_1^4\eta_3^3\eta_4^3 \leq B}
    \eta_1^{2/3} 2^{\omega(\eta_1\eta_3\eta_4)} \widetilde{X_5}^{1/3+\varepsilon} \\
    &\ll B^{1/2+\varepsilon}\sum_{\eta_1^4\eta_3^3\eta_4^3 \leq B}
    \frac{1}{\eta_3^{1/2}\eta_4^{1/2}}  \ll B^{3/4 + \varepsilon}
\end{align*}
which is satisfactory. Now we can finally make the main term of
$T_1$ explicit, which in the case $ (\eta_1,\eta_3,\eta_4) \in
\mathcal{N}_0$ is
\begin{align*}
    B\sum_{{\eta_1,\eta_3,\eta_4 \in \mathcal{N}_0}}
    \widetilde{V}(1)\frac{\psi(\eta_1,\eta_3,\eta_4)}{\eta_1^2\eta_3^2\eta_4^2}
    \sum_{k_1\mid\eta_1\eta_3\eta_4}\mu(k_1)
    \int_0^1u^{1/3}\widetilde{E}(u\widetilde{X_5})\mathrm{d}u.
\end{align*}
We know that $\widetilde{E}(u\widetilde{X_5})\leq 6/u^{1/3}$ is
actually independent of $B$, so letting the sum over the $\eta_i$ go
to infinity, we get a main term of the form $\lambda'B$ where
$\lambda' \in \mathbb{R}$ is some constant and an error term of the
order
\begin{align*}
    &\ll B^{1+\varepsilon}\sum_{\eta_1^4\eta_3^3\eta_4^3 > B}
    \frac{1}{\eta_1^2\eta_3^2\eta_4^2}
    \ll B^{3/4+\varepsilon}
\end{align*}
which is satisfactory. This was only for the case
$(\eta_1,\eta_3,\eta_4) \in \mathcal{N}_0$, however it is clear that
the sum over the case where $(\eta_1,\eta_3,\eta_4) \in
\mathcal{N}_1$ is almost exactly the same and hence it is omitted.
So returning to the original problem (\ref{def:T1}), we have shown
that there exists a constant $\lambda \in \mathbb{R}$ such that
$$T_1(B)=\lambda B + O(B^{3/4+\varepsilon}).$$

\subsection{Summation over the $\eta_i$}
We now know that
\begin{align*}T(B)=&\sum_{\bfeta \in \mathcal{N}}
        \frac{\vartheta(\bfeta) F_2(X_5)}{\eta_4X_3X_6} + \lambda B + O(B^{3/4+\varepsilon}),
\end{align*}
where $X_3,X_5,X_6$ and $\mathcal{N}$ are given by (\ref{height1})
and (\ref{bigeta}), $F_2$ is as in Lemma~\ref{lem:F2E}, and we
define

$$\vartheta(\bfeta)=\sum_{\substack{k_3\mid\eta_1 \\
(k_3,\eta_2\eta_3)=1}}\frac{\mu(k_3)}{k_3}
\sum_{\substack{k_2\mid\eta_1\eta_2 \\ (k_2,k_3\eta_4)=1}}
\frac{\mu(k_2)}{k_2}
\sum_{k_1\mid\eta_1\eta_3\eta_4}\frac{\mu(k_1)}{k_1}$$ when
$(\eta_2,\eta_3)=(\eta_2,\eta_4)=(\eta_3,\eta_4)=1$ and
$\vartheta(\bfeta) = 0$ otherwise. We have already simplified a very
similar sum in Lemma~\ref{lem:nu}, and using a similar method one
can check that
\begin{align}
    \vartheta(\bfeta) & =\phi^*(\eta_1)\phi^*(\eta_2)\phi^*(\eta_3)\phi^*(\eta_4)
    \prod_{\substack{p\mid \eta_1 \\ p \nmid \eta_2\eta_3\eta_4}}
    \left(1-\frac{2}{p}\right) \label{vartheta}
\end{align}
when $(\eta_2,\eta_3)=(\eta_2,\eta_4)=(\eta_3,\eta_4)=1$ and
$\vartheta(\bfeta) = 0$ otherwise. Recalling the height conditions
(\ref{height1}) it follows that

\begin{align*}
        T(B)&
         =B^{2/3}\sum_{n \leq B}\Delta(n) F_2\left(\left(\frac{n}{B}\right)^{1/3}\right)
          + \lambda B + O(B^{3/4+\varepsilon})
\end{align*}
where
\begin{equation}
\Delta(n)=\sum_{\eta_1^4\eta_2^2\eta_3^3\eta_4^3=n}
\vartheta(\bfeta) \left(\frac{\eta_1}{\eta_2}\right)^{1/3}.
\label{Delta}
\end{equation}
Hence we have the expression
\begin{equation}\label{N(B)}
N_{U,H}(B)=2B^{2/3}\sum_{n \leq B}\Delta(n)
F_2\left(\left(\frac{n}{B}\right)^{1/3}\right)
+\left(\frac{12}{\pi^2} + 2\lambda\right)B + O(B^{3/4+\varepsilon})
\end{equation}
for the counting function.

\subsection{The Height Zeta Function}
In this section we shall prove Theorem~\ref{thm:HZF} on the height
zeta function $Z_{U,H}(s)$ as defined in (\ref{HZF}). A standard
application of Perron's formula \cite[Lemma 3.12]{Tit86} gives us an
expression for the counting function $N_{U,H}(B)$ in terms of the
zeta function via an inverse Mellin transform. Then performing the
corresponding Mellin transform tells us that for $\re (s)\gg1$ we
have
\begin{equation}
    Z_{U,H}(s)=s\int_1^\infty u^{-s-1}N_{U,H}(u)\mathrm{d}u \label{zeta
    Mellin}
\end{equation}
where $s=\sigma + it$ is a complex variable. Recalling (\ref{N(B)}),
we have $Z_{U,H}(s) = Z_1(s) + Z_2(s)$ where
\begin{align}
    Z_1(s) &= 2s\int_1^\infty u^{-s-1/3} \sum_{n \leq u}\Delta(n)
    F_2\left(\left(\frac{n}{u}\right)^{1/3}\right) \mathrm{d}u, \nonumber\\
    Z_2(s) &= \frac{12/\pi^2 + 2\lambda}{s-1} + G_2(s),\label{G2}\\
    G_2(s) &= s\int_1^\infty u^{-s-1}R(u)\mathrm{d}u,\nonumber
\end{align}
and $R(u)$ is some function such that $R(u)\ll u^{3/4+\varepsilon}$
for all $\varepsilon>0$. From this it follows that $G_2(s)$ is
holomorphic on the half-plane $\re(s)\geq 3/4+\varepsilon$, and
moreover
\begin{align*}
    G_2(s)& \ll|s|\int_1^\infty u^{-\sigma -1}u^{3/4+\varepsilon}\mathrm{d}u
    \ll\frac{|1 + i\frac{t}{\sigma}|}{|\frac{3}{4\sigma}-1|}
    \ll 1 + |t|
\end{align*}
on this domain (note that here we use the common abuse of notation
that $\varepsilon$ is allowed to take different values
simultaneously). In particular $Z_2(s)$ has a meromorphic
continuation to the same half-plane with a simple pole at $s=1$ of
residue $12/\pi^2 + 2\lambda$.

Now that $Z_2(s)$ is under control, let us turn our attention to
$Z_1(s)$. Define $\Delta$'s Dirichlet series by
$D(s)=\sum_{n=1}^\infty \Delta(n) n^{-s}$. Then by choosing a
suitable $s$ to make sure that change of sum and integral are valid,
we can simplify $Z_1$  by
\begin{align}
    Z_1(s) &= 2s\sum_{n=1}^\infty\Delta(n)\int_n^\infty u^{-s-1/3}
    F_2\left(\left(\frac{n}{u}\right)^{1/3}\right) \mathrm{d}u \nonumber \\
    &=2sD\left(s-\frac{2}{3}\right)\int_1^\infty u^{-s-1/3}
    F_2\left(\left(\frac{1}{u}\right)^{1/3}\right) \mathrm{d}u \nonumber \\
    &=D\left(s-\frac{2}{3}\right)G_{1,1}(s), \label{Z1}
\end{align}
where
\begin{align}
G_{1,1}(s) & = 6s\int_0^1 u^{3(s-1)}F_2(u)\mathrm{d}u\label{G11}.
\end{align}
A standard application of \cite[Lemma 4.3]{Tit86} combined with
(\ref{bounds:F2E}) now tells us that $G_{1,1}(s)$ is a bounded and
holomorphic function on the half-plane $\re(s) > 5/6$. Recalling the
definition of $\Delta$ in (\ref{Delta}), we find that
\begin{align*}
    D(s + 1/3)&=
    \sum_{\eta_1,\eta_2,\eta_3,\eta_4=1}^\infty\frac{\vartheta(\eta_1,\eta_2,\eta_3,\eta_4)}
    {\eta_1^{4s+1}\eta_2^{2s+1}\eta_3^{3s+1}\eta_4^{3s+1}}\\
    &= \prod_p \sum_{k_i=0}^\infty
    \frac{\vartheta(p^{k_1},p^{k_2},p^{k_3},p^{k_4})}{p^{(4s+1)k_1 +
    (2s+1)k_2 + (3s+1)(k_3+k_4)}}.
\end{align*}
After recalling the expression for $\vartheta(\bfeta)$ in
(\ref{vartheta}) and using the fact that $\vartheta(\bfeta) \neq 0$
if and only if $(n_2,n_3)=(n_2,n_4)=(n_3,n_4)=1$, this sum greatly
simplifies and it is easy to see that  $D(s + 1/3) = \prod_p D_p(s +
1/3)$ where
\begin{align*}
    D_p(s + 1/3)  =1 +
    &\left(1 - \frac{1}{p} \right)
    \left(\frac{1}{p^{2s+1}-1} +
    \frac{2}{p^{3s+1}-1} +
    \frac{1-2/p}{p^{4s+1}-1}
    \right) \\
    +&
    \left(1-\frac{1}{p}\right)^2
    \left(\frac{1}{p^{4s+1}-1}\right)
    \left(\frac{1}{p^{2s+1}-1} +
    \frac{2}{p^{3s+1}-1}
    \right).
\end{align*}
Recalling the definition of $E_1(s)$ and $E_2(s)$ given by
(\ref{E12}), we can prove the following.

\begin{lem} \label{lem:G12}
    We have
        $$D(s+1/3) = E_1(s+1)E_2(s+1)G_{1,2}(s+1)$$
    where $G_{1,2}(s+1)$ is holomorphic and bounded on the half-plane
    $\mathcal{H}=\{s \in \mathbb{C} : \re(s)\geq-1/3 + \varepsilon\}$.
\end{lem}
\begin{proof}
    Defining $G_{1,2}(s+1)=D(s+1/3)/(E_1(s+1)E_2(s+1))$, it is clear that it will be
    enough to show that $G_{1,2}(s+1) = \prod_p(1 + O(1/p^{1+\varepsilon}))$
    on $\mathcal{H}$. A routine calculation tells us that
    \begin{align*}
        D_p(s+1/3)&\left(1 - \frac{1}{p^{4s+1}}\right) =
        1 - \frac{3}{p^{4s+2}} + \frac{2}{p^{4s+3}} \\
        &+\left(1-\frac{1}{p}\right)\left(1 - \frac{1}{p^{4s+2}}\right)\left(\frac{1}{p^{2s+1}-1} + \frac{2}{p^{3s+1}-1}
        \right).
    \end{align*}
    Now on $\mathcal{H}$ we have the following estimates
    \begin{align*}
        \frac{1}{p^{4s+2}} & =
        O\left(\frac{1}{p^{2/3+\varepsilon}}\right) , \quad
        \frac{1}{p^{2s+1}-1} =
        O\left(\frac{1}{p^{1/3+\varepsilon}}\right), \\
        \frac{1}{p^{4s+3}} & =
        O\left(\frac{1}{p^{5/3+\varepsilon}}\right), \quad
        \frac{1}{p^{3s+1}-1} =
        O\left(\frac{1}{p^{\varepsilon}}\right).
    \end{align*}
    So on $\mathcal{H}$ we have
    \begin{align*}
        D_p(s+1/3)\left(1 - \frac{1}{p^{4s+1}}\right) &=
        1 - \frac{3}{p^{4s+2}} + \frac{1}{p^{2s+1}-1} \\
        &+\frac{2}{p^{3s+1}-1}\left(1 - \frac{1}{p^{4s+2}}\right) +
        O\left(\frac{1}{p^{1+\varepsilon}}\right).
    \end{align*}
    And finally an easy calculation gives us
    \begin{align*}
        \frac{D_p(s+1/3)}{E_{1,p}(s+1)}= &1 - \frac{3}{p^{4s+2}}
        - \frac{2}{p^{5s+2}} - \frac{1}{p^{6s+2}} +
        \frac{4}{p^{7s+3}} \\
        & \quad+ \frac{2}{p^{8s+3}} - \frac{1}{p^{10s+4}}
        + O\left(\frac{1}{p^{1+\varepsilon}}\right)
    \end{align*}
    where $E_{1,p}(s+1)$ is the corresponding Euler factor of
    $E_1(s+1)$, thus proving the claim.
\end{proof}
Thus letting
\begin{align}
    G_1(s)=G_{1,1}(s)G_{1,2}(s) \label{G1}
\end{align}
and combining (\ref{G2}),(\ref{Z1}) and (\ref{G11}) with Lemma
\ref{lem:G12}, we have proved Theorem~\ref{thm:HZF}.
\subsection{The Asymptotic Formula}
In this section we shall prove Theorem~\ref{thm:asym}. Our starting
point is the expression for the counting function given by
(\ref{N(B)}), which we shall simplify using partial summation and
the properties of the Dirichlet series $D(s)$ deduced in
Lemma~\ref{lem:G12}. In what follows let $M(B) = \sum_{n \leq B}
\Delta(B)$.

\begin{lem}\label{lem:sum_Delta}
    We have
    $$M(B) = \frac{E_2(1)G_{1,2}(1)}{144}B^{1/3}Q(\log
    B) + O(B^{7/8-2/3 + \varepsilon})$$
    where $Q \in \mathbb{R}[x]$ is some monic cubic polynomial.
\end{lem}
\begin{proof}
    Letting  $T \in [1,B]$, Perron's formula \cite[Theorem 3.12]{Tit86} tells us that for
    non-integral $B$ we have
    $$M(B) = \frac{1}{2\pi
    i}\int_{1/3 + \varepsilon-iT}^{1/3 +
    \varepsilon+iT}D(s)\frac{B^s}{s}ds + O\left(\frac{B^{1+\varepsilon}}{T}\right).$$
    Changing variables and using Lemma~\ref{lem:G12} we deduce that
    $$M(B) = \frac{1}{2\pi
    iB^{2/3}}\int_{1 + \varepsilon-iT}^{1 +
    \varepsilon+iT}E_1(s)E_2(s)G_{1,2}(s)\frac{B^s}{s-2/3}ds +
    O\left(\frac{B^{1+\varepsilon}}{T}\right).$$
    Now let $a \in [7/8,1)$ and let $\Gamma$ be the rectangular contour
    through the points $a-iT,a+iT,1+\varepsilon - iT,1+\varepsilon +
    iT$. Then, as we have already shown, $E_2(s)$ and $G_{1,2}(s)$ are
    holomorphic and bounded inside this contour, and $E_1(s)$ has a pole
    of order $4$ at $s=1$. Recalling that $\zeta(s)$ has a simple pole
    of order $1$ at $s=1$ with residue $1$, we have
    $\lim_{s\to1}E_1(s)(s-1)^4=4\cdot3^2\cdot2=72$. Also we have the
    following Taylor series
    $$B^s=B\sum_{n=1}^\infty \frac{(\log B)^n (s-1)^n}{n!}$$
    which gives us the residue
    $$\mbox{Res}_{s=1}\left\{E_1(s)E_2(s)G_{1,2}(s)\frac{B^s}{s-2/3}\right\}
    = \frac{E_2(1)G_{1,2}(1)}{144}BQ(\log B)$$ where $Q \in \mathbb{R}[x]$
    is some monic cubic polynomial. So letting $$\mathcal{E}(s)
    =\sum_{n \leq B} \Delta(n) - \frac{E_2(1)G_{1,2}(1)}{144}B^{1/3}Q(\log
    B)$$ and applying Cauchy's
    residue theorem to the contour $\Gamma$, we deduce that

    \begin{align*}
    \mathcal{E}(s) &\ll B^{-2/3}\left(\int_{a-iT}^{a+iT} +
    \int_{a-iT}^{1+\varepsilon-iT} +
    \int_{1+\varepsilon+iT}^{a+iT}\right)
    \left|E_1(s)\frac{B^s}{s}\right|ds + \frac{B^{1+\varepsilon}}{T}.
    \end{align*}
    From \cite[Ch. II.3.4, Theorem 6]{Ten95} we have the bound
    $$\zeta(\sigma + it) \ll
     |t|^{(1-\sigma)/3 + \varepsilon}, \quad \mbox{ if } \sigma\in [1/2,1]. $$
    Note that our choice of $a$ implies that in the strip $a < \re(s) <
    1$, we have $4\sigma-3, 3\sigma-2, 2\sigma-1
    >1/2$, so $|E_1(s)| \ll |t|^{4(1-\sigma)+\varepsilon}.$ Then the
    contribution from the first horizontal contour is
    \begin{align*}
        \int_{a-iT}^{1+\varepsilon-iT}\left|E_1(s)\frac{B^s}{s}\right|ds
        & \ll \int_{a}^{1+\varepsilon}T^{3-4\sigma+\varepsilon}B^{\sigma}d\sigma
        \\
        &\ll \frac{B^{1+\varepsilon}T^{\varepsilon}}{T} +
        B^aT^{3-4a+\varepsilon},
    \end{align*}
    and the same bound is obtained for the other horizontal contour. For
    the vertical contour we will use well-known estimates for the fourth
    moment of the zeta function. First note that
    $$\int_{a-iT}^{a+iT} \left|E_1(s)\frac{B^s}{s}\right|ds
    \ll B^a\int_{-T}^T \frac{|E_1(a+it)|}{1+|t|}\mathrm{d}t.$$ Now let
    $0 < U \ll T$ and consider the following dyadic interval
    $$\int_{U}^{2U} \frac{|E_1(a+it)|}{1+|t|}\mathrm{d}t
    \ll \frac{1}{U}\int_{U}^{2U} |E_1(a+it)|\mathrm{d}t =
    \frac{J(U)}{U},$$ say. H\"{o}lder's inequality now tells us that
    $$J(U) \leq J_4(U)^{1/4} J_3(U)^{1/2} J_2(U)^{1/4}$$
    where $J_k(U)=\int_{U}^{2U}|\zeta(k(a -1)+1+kit)|^4\mathrm{d}t$. Now
    by convexity \cite[Ch. VII.8]{Tit86} and the fact that we have
    $\int_0^T|\zeta(1/2+it)|^4 \ll T\log^4T$ by \cite[Th. 1]{HB79}, we
    see that for $\sigma\in [1/2,1]$ we have
    $$\int_U^{2U}|\zeta(\sigma+it)|^4\mathrm{d}t \ll U^{1+\varepsilon}.$$
    Hence we deduce that $J(U) \ll U^{1+\varepsilon}$. Now summing over
    these dyadic intervals we find
    $$\int_0^T \frac{|E_1(a+it)|}{1+|t|}\mathrm{d}t \ll
    T^{\varepsilon}.$$ The same estimate holds over the interval
    $[-T,0]$, and so putting everything together we find an overall
    error of $$\mathcal{E}(s) \ll \frac{B^{1+\varepsilon}}{T} + B^{a -2/3 +
    \varepsilon}.$$ Taking $T=B,a=7/8+\varepsilon$, the
    error we obtain is satisfactory for the
    lemma.
\end{proof}
Using this lemma we can deduce the following.

\begin{lem}\label{lem:sum_Delta_F2}
    We have
    \begin{align*}
        &\sum_{n \leq B}\Delta(n) F_2\left(\left(\frac{n}{B}\right)^{1/3}\right) \\
        & =\frac{E_2(1)G_{1,2}(1)}{144}\left(\int_0^1F_2(u)\mathrm{d}u\right) B^{1/3}P(\log B) +
        O(B^{7/8-2/3 + \varepsilon})
    \end{align*}
    where $P \in \mathbb{R}[x]$ is some monic cubic polynomial.
\end{lem}
\begin{proof}
For ease of notation let $C=E_2(1)G_{1,2}(1)/144$. Applying partial
summation, using (\ref{bounds:F2E}) and Lemma~\ref{lem:sum_Delta} we
deduce that
\begin{align*}
    &\sum_{n \leq B}\Delta(n)
    F_2\left(\left(\frac{n}{B}\right)^{1/3}\right) \\
    & =F_2(1)M(B) - \int_1^B M(t)
    \mathrm{d}F_2\left(\left(\frac{t}{B}\right)^{1/3}\right) \\
    & =C\int_1^B F_2\left(\left(\frac{t}{B}\right)^{1/3}\right)
    \mathrm{d}\left(t^{1/3}Q(\log t)\right) +
    O\left(B^{7/8-2/3+\varepsilon}\right)\\
\end{align*}

It remains to simplify the main term. In what follows we focus on
the leading term of the polynomial $Q$, the lower order terms being
dealt with similarly. After changing variables we deduce that it
equals
\begin{align*}
    & CB^{1/3}\int_{1/B^{1/3}}^1 F_2(u)
    \mathrm{d}\left(u (\log u^3B)^3\right) \\
    & =CB^{1/3}(\log B)^3\int_{B^{-1/3}}^1 F_2(u)
    \mathrm{d}u + \cdots,
\end{align*}
where all the implied lower order terms are easily seen to be of the
order $O(B^{1/3}(\log B)^2\int_{B^{-1/3}}^1 F_2(u) u^{\varepsilon}
\mathrm{d}u)$. On using (\ref{bounds:F2E}) to deduce that
$$\int_0^{B^{-1/3}} F_2(u)u^\varepsilon\mathrm{d}u \ll B^{-1/6 + \varepsilon},$$
the result follows.
\end{proof}

Hence, combing Lemma~\ref{lem:sum_Delta_F2} with (\ref{N(B)}), we
deduce the asymptotic formula given in Theorem~\ref{thm:asym}. One
can also verify the leading constant, after noticing that
$\tau_{\infty}(\widetilde{S})=6\int_0^1 F_2(u) \mathrm{d}u$ and
using Lemma~\ref{lem:G12} to deduce that
$\prod_p\tau_p(\widetilde{S}) = E_2(1)G_{1,2}(1)$.

\end{document}